\documentclass[letterpaper, 10 pt, conference]{ieeeconf}  

\IEEEoverridecommandlockouts                              

\usepackage{graphicx} 
\usepackage{algorithm} 
\usepackage{algpseudocode}
\usepackage{amsmath}

\usepackage{amsthm}
\usepackage{amsfonts} 
\usepackage{caption}
\usepackage{url} 
\usepackage{hyperref}
\usepackage[utf8]{inputenc}

\usepackage{tikz}

\newtheorem{definition}{Definition}
\newtheorem{proposition}{Proposition}
\newtheorem{problem}{Problem}

\newtheorem{theorem}{Theorem}

\usepackage{amsmath,amssymb,amsfonts}
\usepackage{amsmath} 
\usepackage{xcolor}    
\usepackage[algo2e,ruled,linesnumbered,vlined,boxed,commentsnumbered]{algorithm2e}
\title{\LARGE \bf
Model Predictive Online Monitoring of Dynamical Systems for\\
Nested Signal Temporal Logic Specifications
}

\author{Tao Han,  Shaoyuan Li and Xiang Yin
\thanks{This work was supported by the National Natural Science Foundation of China (62173226, 62061136004).}
 \thanks{T. Han,  S. Li and X. Yin are with the School of Automation and Intelligent Sensing, Shanghai Jiao Tong University, Shanghai 200240, China.
	{\tt\small  E-mail: \{coolhantao,syli,yinxiang\}@sjtu.edu.cn}.
}
}

\begin{document}

\maketitle
 
\thispagestyle{empty}
\pagestyle{empty}
\setlength{\abovecaptionskip}{6pt}
\setlength{\belowcaptionskip}{3pt}
\setlength{\textfloatsep}{0pt}
 \def \init{{\scriptstyle\textsf{int}}}
 \def \endt{{\scriptstyle\textsf{end}}}
\begin{abstract}
This paper investigates the online monitoring problem for cyber-physical systems under signal temporal logic (STL) specifications. 
The objective is to design an online monitor that evaluates system correctness at runtime based on partial signal observations up to the current time  so that alarms can be issued whenever the specification is violated or will inevitably be violated in the future.
We consider a model-predictive setting where the system's dynamic model is available and can be leveraged to enhance monitoring accuracy. 
However, existing approaches are limited to a restricted class of STL formulae, permitting only a single application of temporal operators. This work addresses the challenge of nested temporal operators in the design of model-predictive monitors.
Our method utilizes syntax tree structures to resolve dependencies between temporal operators and introduces the concept of basic satisfaction vectors. 
A new model-predictive monitoring algorithm is proposed by recursively updating these vectors online while incorporating pre-computed satisfaction regions derived from offline model analysis. We prove that the proposed approach is both sound and complete, ensuring no false alarms or missed alarms.
Case studies are provided to demonstrate the effectiveness of our method. 
\end{abstract}

\section{Introduction} 

Specification-based monitoring  has emerged as a popular approach for evaluating the safety and correctness of complex engineering cyber-physical systems, such as smart cities \cite{ma2021novel}, autonomous vehicles \cite{sahin2020autonomous}, power plants \cite{beg2018signal}, and industrial IoTs \cite{chen2020temporal}.
In this context, a monitor observes the state trajectory generated by the system and evaluates the correctness of the trajectory based on a given formal specification \cite{bartocci2018specification,yan2022distributed, bonnah2022runtime}. Compared with formal verification techniques such as model checking, the key advantage of monitoring is that it is more lightweight, as it only assesses the correctness of the system along a single observed trajectory without explicitly enumerating the entire reachable state space. Therefore, monitoring can be designed independently and implemented as an add-on module for arbitrary systems—even those treated as black boxes.

In recent years, significant advancements have been made in designing monitoring algorithms for various types of formal specifications, such as Linear Temporal Logic (LTL) \cite{eisner2003reasoning}, Metric Temporal Logic (MTL) \cite{dokhanchi2014line,donze2013efficient}, and Signal Temporal Logic (STL) \cite{deshmukh2017robust,zhang2023online}. 
Among these, STL has become one of the most widely used formal specifications for CPS due to its ability to characterize complex spatio-temporal constraints in real-valued, real-time physical signals.
Monitoring algorithms can be further categorized into offline monitoring \cite{fainekos2009robustness,donze2013efficient} and online monitoring \cite{dokhanchi2014line, deshmukh2017robust}, depending on the information available to the monitor. In the offline setting, the monitor evaluates the correctness of a system using the entire trajectory. However, this approach is unsuitable for systems operating in real time, where only the signal generated up to the current moment is accessible.
In contrast, online monitoring algorithms must account for all possible future behaviors to issue early warnings or terminate processes before a specification violation occurs. As a result, online monitoring is widely adopted in safety-critical systems as a predictive add-on component to ensure operational safety.

Due to the physical dynamics of the system, online signals are not arbitrary in general but follow underlying rules. 
By leveraging information about system behavior, we can better reason about the feasibility or likelihood of future signals, enabling more precise monitoring decisions \cite{abbas2022leveraging}. 
To this end, \emph{predictive monitoring} has recently emerged for systems with either an explicit model or a large dataset of operational history. 
For example, in \cite{qin2020clairvoyant, ma2021predictive, yoon2021predictive, momtaz2023predicate}, 
neural networks trained on historical data are used to predict future signals in order to improve monitoring accuracy.
However, since these approaches adopt a data-driven methodology, their predictions may be unreliable. To address this, uncertainty quantification techniques, such as conformal predictions, are often employed to ensure prediction confidence \cite{cairoli2023learning,lindemann2023conformal,zhao2024robust}. While effective, these methods may introduce additional conservatism into the final monitoring results.

More recently, by leveraging the explicit system model of the underlying dynamical system, a new framework called \emph{model-predictive monitoring} has been proposed in \cite{yu2024model}. This approach further improves monitoring accuracy when a precise system model is available. By integrating offline reachability computations with online satisfaction evaluation, it achieves an effective balance between computational complexity and accuracy. In \cite{wang2024sleep}, a self-triggered information acquisition mechanism is introduced for model-predictive monitoring to reduce state-sampling overhead. However, existing model-predictive monitoring algorithms remain limited to a simple fragment of STL specifications and nested temporal operators, e.g., ``eventually always stays in a region",  are currently unsupported.

In this work, we investigate the design of model-predictive online monitors for signal temporal logic specifications. 
Unlike existing approaches, we consider a general fragment of STL that permits nested temporal operators.
This nested setting introduces fundamental challenges beyond prior works, which 
track STL formula progress using an index set of remaining formulae. 
However, this approach fails under nested temporal constraints due to dependencies between operators. To address this, we propose to use syntax trees \cite{ho2022automaton,yu2024continuous,claudet2024novel} to resolve temporal operator dependencies, and use basic satisfaction vectors as dynamically updated key information during online monitoring.
Our algorithm leverages these vectors, updated recursively in real time, along with pre-computed safety regions derived from offline model analysis. 
This framework successfully extends model-predictive monitoring to general STL specifications. Case studies are also provided to demonstrate the effectiveness of our approach.

The remainder of the paper is organized as follows. 
In Section~\ref{sec:2}, we first present some basic preliminaries. 
The model-predictive monitoring problem is then formulated in Section~\ref{sec:3}. 
In Section~\ref{sec:4}, we present the syntax tree structure and show how to store 
the progress of nested STL formulae using basic satisfaction vectors. 
The main online monitoring algorithm is provided in Section~\ref{sec:5}, where the offline pre-computation of satisfaction regions is also discussed.
The proposed results are illustrated by case studied in Section~\ref{sec:6}. 
Finally, we conclude the paper in Section~\ref{sec:7}.

\section{Preliminary}\label{sec:2} 
\subsection{System Model}
We consider a discrete-time control system of form
\begin{equation}\label{system}
x_{k+1}=f(x_k,u_k),
\end{equation}
where 
$x_{k}\in\mathcal{X}\subset \mathbb{R}^{n}$ denotes the system state at time instant $ k \in \mathbb{Z}_{\geq 0} $, 
$u_{k}\in\mathcal{U}\subset \mathbb{R}^{m}$ represents the control input at time instant $k$ and $f:\mathcal{X}\times\mathcal{U}\rightarrow\mathcal{X}$ is the dynamic function of the system, assumed to be continuous in $\mathcal{X}\times\mathcal{U}$. 
We assume that the state space $\mathcal{X}$ and input space $\mathcal{U}$ are both bounded due to physical constraints.

Suppose that the system state is $x_k$ at time instant $k$. 
Then given a sequence of control inputs $\mathbf{u}_{k:T-1} = u_k u_{k+1} \dots u_{T-1} \in \mathcal{U}^{T-k} $, the corresponding state trajectory generated by the system is defined as $ \xi_f(x_k, \mathbf{u}_{k:T-1}) = \mathbf{x}_{k+1:T} = x_{k+1} \dots x_T \in \mathcal{X}^{T-k} $, where each subsequent state satisfies the recursive relation $x_{i+1} = f(x_i, u_i)$ for all $i = k, \ldots, T-1$.
 
\subsection{Signal Temporal Logic}
We adopt Signal Temporal Logic (STL) as the formal specification language to evaluate trajectory correctness. The syntax of STL formulae is recursively defined as:
\[
\Phi::=\top\mid\pi^\mu\mid\neg\Phi\mid\Phi_1\wedge\Phi_2\mid\Phi_1\mathbf{U}_{[a,b]}\Phi_2, 
\]
where 
$\top$ is the \textit{true} predicate and 
$\pi^{\mu}$ is an atomic predicate whose truth value is determined by the sign of its underlying predicate function $\mu:\mathbb{R}^{n}\rightarrow\mathbb{R}$. 
Specifically, an atomic predicate $\pi^{\mu}$ is true at state $x_{k}$ when $\mu(x_{k}) \geq 0$; otherwise it is false. Operators $\neg$ and $\wedge$ are the standard Boolean operators ``negation" and "conjunction", respectively. 
One can further use them to induce other operators such as ``disjunction" 
$\Phi_{1} \vee \Phi_{2} := \neg(\neg \Phi_{1} \wedge \neg \Phi_{2})$ and ``implication"  $\Phi_{1} \rightarrow \Phi_{2} := \neg \Phi_{1} \vee \Phi_{2}$. $\mathbf{U}_{[a,b]}$ is the temporal operator ``until", where $a, b \in \mathbb{Z}_{\geq 0}$ are two integers with $a \leq b$. 
Note that, since we consider discrete-time setting,  $[a,b]$  is the set of all integers between $a$ and $b$ including themselves. 


Let  $\mathbf{x} = x_0 x_1 \dots$ be a state sequence,  $ k \in \mathbb{Z}_{\geq 0} $ be a time instant and $\Phi$ be an STL formula. 
We denote by  $(\mathbf{x},k) \models \Phi$ 
if sequence $\mathbf{x}$ satisfies STL formula $\Phi$ at time instant $k$. 
The reader is referred to \cite{malerMonitoringTemporalProperties2004} for more details on the semantics of STL formulae. Particularly, for atomic predicates,, we have 
$(\mathbf{x}, k) \models \pi^{\mu}$ iff 
$\mu(x_k) \geq 0$, i.e., $\mu(x_k)$ is non-negative for the current state $x_k$, 
and for temporal operators, 
we have 
$(\mathbf{x},k) \models \Phi_{1} \mathbf{U}_{[a,b]}\Phi_{2}$ iff 
there exists $ k'\!\in\![k+a, k+b]$ such that 
\[
(\mathbf{x}, k')\models \Phi_{2}\wedge (\forall k''\!\in\![k, k']) [(\mathbf{x}, k'') \models \Phi_{1}], 
\]
i.e., 
$\Phi_{2}$ will eventually be satisfied at some instant between $[k+a, k+b]$ 
and $\Phi_1$ holds consistently before then. Furthermore, we can also induce temporal operators:

\begin{itemize}
    \item “\textit{eventually}" $\mathbf{F}_{[a,b]}\Phi := \top \mathbf{U}_{[a,b]}\Phi$ such that it holds when $(\mathbf{x},k) \models \Phi$ for some $k' \in [k+a, k+b]$;
    \item “\textit{always}" $\mathbf{G}_{[a,b]}\Phi := \neg \mathbf{F}_{[a,b]} \neg \Phi$ such that it holds when $(\mathbf{x},k) \models \Phi$ for any $k' \in [k+a, k+b]$.
\end{itemize}
We write $\mathbf{x} \models \Phi$ whenever $(\mathbf{x}, 0) \models \Phi$. 
We assume that the operation horizon $T$ is sufficiently long to evaluate the satisfaction of  $\Phi$. 

\section{Problem Formulation}\label{sec:3}

\subsection{Fragment of STL Formulae}
In this paper, 
we consider the following slightly restricted fragments of STL formulae
\begin{subequations}\label{eq:myequations}
\begin{equation}\varphi::=\top\mid\pi^\mu\mid\neg\varphi\mid\varphi_1\wedge\varphi_2\end{equation}
\begin{equation}\Phi::=\mathbf{F}_{[a,b]}\Phi\mid\mathbf{G}_{[a,b]}\Phi\mid\Phi_{1}\mathbf{U}_{[a,b]}\Phi_{2}\mid\Phi_{1}\wedge\Phi_{2}\mid\varphi. 
\end{equation}
\end{subequations}
Note that, compared to the full fragment of STL, we require that negation can only be applied to Boolean operators. Therefore, the overall STL formula considered is the conjunction of a set of sub-formulae in which nested temporal operators can be applied to an arbitrary Boolean formula.  

For technical purposes, we introduce a new temporal operator \(\mathbf{U}^{\prime}\), defined by: \((\mathbf{x},k) \models \Phi_1\mathbf{U}_{[a,b]}^{\prime}\Phi_2\) iff there exists \(k' \in [k+a,k+b]\) such that  
\[  
(\mathbf{x},k') \models \Phi_2 \wedge (\forall k''\in [k+a,k'])[(\mathbf{x},k'') \models \Phi_1].  
\]  
Compared to the original definition of ``until," the key difference is that the effective horizon of \(\mathbf{U}_{[a,b]}\) is \([0, b]\), while the effective horizon of \(\mathbf{U}_{[a,b]}^{\prime}\) is \([a, b]\). Throughout this paper, we will refer to \(\mathbf{U}_{[a,b]}^{\prime}\) as the ``until" operator. This replacement is primarily technical and does not lose generality, as the standard \(\mathbf{U}\) can always be expressed as  
\[  
(\mathbf{x},k) \models \Phi_1\mathbf{U}_{[a,b]}\Phi_2 \Leftrightarrow (\mathbf{x},k) \models (\Phi_1\mathbf{U}_{[a,b]}^{\prime}\Phi_2) \wedge (\mathbf{G}_{[0,a]}\Phi_1).  
\]

Recall that for the predicate \(\pi^\mu\) in (2a), its satisfaction region, denoted by \(\mathcal{H}^{\mu}\), is the solution to the inequality \(\mu(x) \geq 0\), i.e., \(\mathcal{H}^{\mu} = \{x \in \mathcal{X} \mid \mu(x) \geq 0\}\). For other Boolean operators, we have \(\mathcal{H}^{\neg\varphi} = \mathcal{X} \setminus \mathcal{H}^{\varphi}\) and \(\mathcal{H}^{\varphi_1 \wedge \varphi_2} = \mathcal{H}^{\varphi_1} \cap \mathcal{H}^{\varphi_2}\). Therefore, instead of writing \(\varphi\), we will hereafter simply denote it as \(x \in \mathcal{H}^\mu\) or \(x \in \mathcal{H}\), using its satisfaction region.  

Additionally, while we consider the temporal operator ``eventually" (\(\mathbf{F}\)) in the semantics, it is subsumed by ``until" (\(\mathbf{U}'\)) since \(\mathbf{F}_{[a,b]} \Phi\) can be expressed as \(x \in \mathcal{X} \mathbf{U}_{[a,b]}^{\prime} \Phi\). Thus, we only need to handle the temporal operators \(\mathbf{G}\) and \(\mathbf{U}'\) from a technical standpoint.  

Based on the above discussion, the STL formula \(\Phi\) in ~\eqref{eq:myequations} can be equivalently expressed as:  
\begin{equation}\label{eq:syntax-STL-2}  
\Phi ::= \mathbf{G}_{[a,b]} \Phi \mid \Phi_1 \mathbf{U}_{[a,b]}^{\prime} \Phi_2 \mid \Phi_1 \wedge \Phi_2 \mid x \in \mathcal{H}  
\end{equation}

\subsection{Online Monitoring of STL}
At time instant $k\leq T$, the system only generates a 
partial signal $\mathbf{x}_{0:k}=x_0x_1\dots x_k$ (called a prefix)  and the remaining signals $ \mathbf{x}_{k+1:T} $ (called suffix) will only be available in the future. 
Therefore, we a partial signal $\mathbf{x}_{0:k}$, we denote by 
\begin{itemize}
    \item 
    $\mathbf{x}_{0:k}\models \Phi$ if $\forall \mathbf{x}_{k+1:T}\in \mathcal{X}^{T-k}: 
    \mathbf{x}_{0:k}\mathbf{x}_{k+1:T}\models \Phi$; 
    \item 
    $\mathbf{x}_{0:k}\not\models \Phi$ if $\forall \mathbf{x}_{k+1:T}\in \mathcal{X}^{T-k}: 
    \mathbf{x}_{0:k}\mathbf{x}_{k+1:T}\not\models \Phi$; 
    \item 
    $\mathbf{x}_{0:k}\models_? \Phi$ otherwise. 
\end{itemize}

Note that the above partial signal evaluation does not account for system dynamics. In other words, some suffixes in  $\mathbf{x}_{k+1:T}\in \mathcal{X}^{T-k}$ may be dynamically infeasible. Thus, in the context of \emph{model-predictive monitoring}, we further classify a prefix signal  $ \mathbf{x}_{0: k} $ as
\begin{itemize}
    \item 
    \textbf{violated} if, for any control input $ \mathrm{u}_{k: T-1} $, we have $\mathbf{x}_{0:k}\xi_f(x_k,\mathrm{u}_{k:T-1})\not\models\Phi$; 
    \item 
    \textbf{feasible} if, for some control input $ \mathrm{u}_{k: T-1} $, we have 
    $ \mathbf{x}_{0: k} \xi_{f}(x_{k}, \mathrm{u}_{k: T-1}) \models \Phi$. 
\end{itemize}
Intuitively, a prefix signal is violated when either the current state already violates the specification or when all possible future trajectories will violate it inevitably. For instance, consider a safety specification $\mathbf{G}_{[0,T]} x \in \mathcal{H}$. If the system reaches any state $x_k \notin \mathcal{H}$ for $k < T$, this immediately constitutes a violation. More subtly, even when the current state satisfies the predicate, the prefix is violated if from state $x_k$ there exists no control sequence $\mathrm{u}_{k:T-1}$ that can generate a trajectory $\xi_f(x_k, \mathrm{u}_{k:T-1})$ remaining entirely within $\mathcal{H}$ throughout the remaining horizon. 
This mirrors real-world scenarios such as a vehicle approaching an obstacle at excessive speed. 
Even before physical collision occurs, the system is already violated because no braking action can prevent the impending violation given the current velocity and distance.

\begin{problem}[\bf Model-Predictive Online Monitoring]\upshape
Given a system with dynamic in ~\eqref{system} and an STL formula $\Phi$, design an online monitor
\begin{equation}
    \mathcal{M}:\mathcal{X}^*\to\{\texttt{vio},\texttt{feas}\}
\end{equation}
such that, for any prefix $\mathbf{x}_{0:k}\in \mathcal{X}^*$, we have
\begin{itemize}
    \item 
    $\mathcal{M}(\mathbf{x}_{0:k})=\texttt{vio}$ iff $\mathbf{x}_{0:k}$ is violated; 
    \item 
    
    $\mathcal{M}(\mathbf{x}_{0:k})=\texttt{feas}$ iff $\mathbf{x}_{0:k}$ is feasible. 
\end{itemize}
\end{problem}

\section{Tree of Nested STL Formulae}\label{sec:4}
The main challenge in handling STL  formulae with nested temporal operators lies in the dependency between inner and outer operators. For instance, consider the STL formula $\mathbf{G}_{[a_1,b_1]}\mathbf{F}_{[a_2,b_2]}\varphi$. This formula requires that the inner formula $\mathbf{F}_{[a_2,b_2]}\varphi$ must be satisfied for all time instants between $a_1$ and $b_1$. Specifically, starting from any time instant within $[a_1, b_1]$, the system should satisfy $\varphi$ within the interval $[a_2, b_2]$ from \emph{that point onward}. Consequently, evaluating the correctness of a trajectory requires a total horizon of $b_1 + b_2$.
To address this dependency clearly, we introduce the concept of  \emph{syntax tree} and the notion of  \emph{satisfaction vector}. These tools help systematically resolve the dependencies between nested temporal operators.

\subsection{Syntax Trees for STL formulae}
Hereafter, we will equivalently represent an STL formula by a \emph{rooted tree}, which is a 4-tuple  
\[
\mathcal{T} = (V,L, E, v_{\text{root}}), 
\]
where:
\begin{itemize}
\item  
$V = \{v_1, v_2, \dots, v_m\}$ is the set of nodes;
\item 
$L = \{l, r, \bot\}$ is a label set, where $l$ and $r$ denote ``left" and ``right", respectively, and $\bot$ denotes ``no order";
\item 
$E \subseteq V \times L \times V$ is the set of edges, where each $(v, \ell, v') \in E$ indicates that $v'$ is a child node of $v$, and is specifically a left or right child if $\ell = l$ or $\ell = r$, respectively;
\item 
$v_{\text{root}} \in V$ is the unique root node with no parent.
\end{itemize}
Clearly, a rooted tree induces a partial order  $\prec$ on $V$, where $v'\prec v$ iff $v$ is an ancestor  of $v'$. 
For each node $v$, we denote by $\textsf{child}(v) = \{v' \mid \exists \ell \in L: (v, \ell, v') \in E\}$ the set of its children. Additionally, we denote by $\textsf{child-l}(v)$ and $\textsf{child-r}(v)$ the left and right child of $v$, respectively, if they exist.


In order to represent an STL formula as a tree, we introduce the following four types of nodes:
\begin{itemize}
  \item 
  $\mathcal{H}$-nodes: 
  Each node represents a satisfaction region of a Boolean formula. 
  Such nodes have no children, as Boolean formulae are always at the innermost level of the entire formula.
  \item 
  $\wedge$-nodes: 
  Each node represents the Boolean operator ``conjunction" and has two or more unordered children. 
  These nodes are evaluated instantly, meaning there is no time interval associated with them.
  \item 
  $\mathbf{G}$-nodes: 
  Each node represents the temporal operator ``always" and has exactly one child. Such nodes are associated with a time interval that determines the evaluation period of their descendants.
  \item 
  $\mathbf{U}'$-nodes: 
  Each node represents the temporal operator ``until" and has both a left and a right child.
  These nodes are also associated with a time interval that determines the evaluation period of their descendants.
\end{itemize}
Now, we are ready to formally define the syntax tree. 
\begin{definition}[\bf Syntax Trees]\upshape 
Let $\Phi$ be an STL formula defined by syntax in ~\eqref{eq:syntax-STL-2}.  
The syntax tree of STL formula $\Phi$, denoted by $\mathcal{T}_{\Phi}=(V_\Phi, L, E_\Phi, v_{\text{root},\Phi})$, 
is defined recursively as follows:
\begin{itemize}
    \item 
    If $\Phi = \Phi_1\wedge\Phi_2\wedge\cdots\wedge\Phi_n$, then 
    $ V_\Phi=(\cup_{i=1}^nV_{\Phi_i})\cup \{v_{\wedge}\}$, 
    $ E_\Phi=(\cup_{i=1}^nE_{\Phi_i})\cup \{ (v_{\wedge},\bot,v_{\text{root},\Phi_i} )\mid i=1,\dots,n\}$, 
    and $v_{\text{root},\Phi}=v_\wedge$ , 
    where $v_\wedge$ is a new $\wedge$-node that does not exist in each $\mathcal{T}_{\Phi_i}$; 
    \item 
    If $\Phi = \mathbf{G}_{[a,b]}\Phi'$, 
    then 
    $ V_\Phi= V_{\Phi'}\cup \{v_{\mathbf{G}}\}$, 
    $ E_\Phi= E_{\Phi'} \cup \{ (v_{\mathbf{G}},\bot,v_{\text{root},\Phi'}) \}$, 
    and $v_{\text{root},\Phi}=v_\mathbf{G}$, 
    where $v_\mathbf{G}$ is a new $\mathbf{G}$-node does not exist in   $\mathcal{T}_{\Phi'}$ and it is associated with time interval $[a,b]$;
    \item 
    If $\Phi = \Phi_1\mathbf{U'}_{[a,b]}\Phi_2$, 
    then 
    $ V_\Phi= V_{\Phi_1}\cup V_{\Phi_2}\cup\{v_{\mathbf{U'}}\}$, 
    $ E_\Phi= E_{\Phi_1} \cup E_{\Phi_2} \cup\{ (v_{\mathbf{G}},l,v_{\text{root},\Phi_1}),(v_{\mathbf{G}},r,v_{\text{root},\Phi_2} )\}$, 
    and $v_{\text{root},\Phi}=v_{\mathbf{U}'}$ , 
    where $v_{\mathbf{U}'}$ is a new $\mathbf{U}'$-node does not exist in $\mathcal{T}_{\Phi_1}$ or $\mathcal{T}_{\Phi_2}$,  and it is also associated with time interval $[a,b]$;
    \item 
    If $\Phi = x \in \mathcal{H}$, then 
    $ \mathcal{T}_\Phi$ only contains a single $\mathcal{H}$-node associated with predicate $x \in \mathcal{H}$. 
\end{itemize}
\end{definition}

For each node $v \in V_\Phi$, we denote by $[a^v, b^v]$ its \textbf{associated   time interval}, with $a^v = b^v = 0$ when $v$ is a $\wedge$-node or an $\mathcal{H}$-node. 
Also, for each $\mathcal{H}$-node $v\in V_\mathcal{H}$, we denote by $\mathcal{H}^v$ the associated predicate region. 
Hereafter, we will omit the subscript $\Phi$ and refer to $\mathcal{T} = (V, L, E, v_{\text{root}})$ as the syntax tree of the formula $\Phi$ when the context is clear. 
For $\star \in \{\wedge, \mathbf{G}, \mathbf{U}', \mathcal{H}\}$, we denote by $V_\star \subseteq V$ the set of all $\star$-nodes in $\mathcal{T}$. Furthermore, let 
\[
V = \{v_1, v_2, \dots, v_{m}\},
\]
and we assume that all $\mathcal{H}$-nodes are ordered as the first $h \geq 1$ elements, i.e., $V_\mathcal{H} = \{v_1, \dots, v_h\}$.

\subsection{Satisfaction Vector}
As discussed, for each node $v \in V$ in the syntax tree, it evaluates its descendants nodes within the time interval $[a^v, b^v]$. Note that this interval is relative to the local perspective, as $v$ may have ancestor nodes in the tree corresponding to outer operators whose time intervals also apply to $v$. Therefore, to determine the absolute time interval for evaluating node $v$ from a global perspective, one must consider the local time intervals of all its  ancestors.

Formally, for each node $v \in V$, we denote by $\textsf{ances}(v) = \{v' \in V \mid v \prec v'\}$ the set of all its ancestors. This set includes all nodes on the unique path from the root node to $v$ in $\mathcal{T}$, excluding $v$ itself. Based on this, we introduce the following definition of \textit{evaluation horizon}.

\begin{definition}[\bf Evaluation Horizons]\upshape
Let $\mathcal{T}$ be the syntax tree of STL formula $\Phi$ and $v\in V$ be a node.
Then the \emph{evaluation horizon} of node $v$  is defined by 
\[\textstyle
[\init^v,\endt^v]=[ \sum_{v'\in \textsf{ances}(v)}  a^{v'}, \sum_{v'\in \textsf{ances}(v)}  b^{v'}   ].
\]
\end{definition}

For example,   consider the following STL formula 
\begin{equation}\label{example}  
\Phi=((\mathbf{G}_{[1,3]}x\in\mathcal{H}_1)\mathbf{U'}_{[2,5]}x\in\mathcal{H}_2)\wedge(\mathbf{G}_{[3,7]}x\in\mathcal{H}_3).
\end{equation}
Its syntax tree $\mathcal{T}$ is shown in Fig.~\ref{fig:exm-tree1}.
For $\mathcal{H}$ node $v_1$, its evaluation horizon is 
$[\init^{v_1},\endt^{v_1}]=[2+1=3,5+3=8]$.

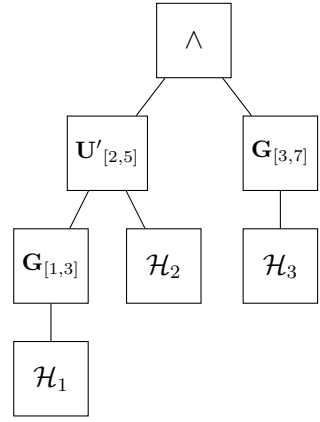
\begin{figure}[t]
    \centering
    \begin{tikzpicture}
    [
        level distance=1.5cm, 
        level 1/.style={sibling distance=2.3cm}, 
        level 2/.style={sibling distance=1.5cm}, 
        level 3/.style={sibling distance=1cm}
    ]
        \node[draw, rectangle, minimum size=1.1cm, scale=0.9, font=\large] {$\wedge$}
            child {node[draw, rectangle, minimum size=1.1cm, scale=0.9] {$\mathbf{U'}_{[2,5]}$}
                child {node[draw, rectangle, minimum size=1.1cm, scale=0.9] {$\mathbf{G}_{[1,3]}$}
                    child {node[draw, rectangle, minimum size=1.1cm, scale=0.9, font=\large] {$\mathcal{H}_1$}}
                }
                child {node[draw, rectangle, minimum size=1.1cm, scale=0.9, font=\large] {$\mathcal{H}_2$}}
            }
            child {node[draw, rectangle, minimum size=1.1cm, scale=0.9] {$\mathbf{G}_{[3,7]}$}
                child {node[draw, rectangle, minimum size=1.1cm, scale=0.9, font=\large] {$\mathcal{H}_3$}}
            };
    \end{tikzpicture}
    \hspace{1cm}
    \caption{Syntax Tree of STL formula ~\eqref{example}}\label{fig:exm-tree1}
\end{figure}

Therefore, for each time instant $t \in [\init^v, \endt^v]$, it is meaningful to discuss the satisfaction status of node $v$. In general, there are three possible satisfaction statuses for a node:
(i) 
\textbf{satisfied}, denoted by ``$1$"; 
(ii) 
\textbf{violated}, denoted by ``$0$"; and 
(iii) 
\textbf{uncertain}, denoted by ``$?$". 
The last case arises when there is insufficient information to evaluate the node. For example, consider the formula $\mathbf{G}_{[10,12]}x \in \mathcal{H}$. 
If the current time instant is $k = 8$, then the satisfaction status of the $\mathcal{H}$-node for $t = 10, 11, 12$ is uncertain because the evaluation horizon of the $\mathcal{H}$-node has not yet been reached. 
However, if at time instant $k = 10$, the current state $x_k$ falls within the region $\mathcal{H}$, then the satisfaction status of the $\mathcal{H}$-node at $t = 10$ becomes $1$. However, the satisfaction status for future instants $t = 11, 12$ remains uncertain.

To capture this, we define the satisfaction status for each time instant within the evaluation horizon of a node as the \emph{satisfaction vector}, as follows.

\begin{definition}[\bf Satisfaction Vectors]\upshape
Let $\mathcal{T}$ be the syntax tree of STL formula $\Phi$ and $v\in V$ be a node with  evaluation horizon $[\init^v, \endt^v]$. 
Then a \emph{satisfaction vector} of node $v$ is a vector of form
\[
\imath=(\imath[\init^v],\imath[\init^v+1],\dots, \imath[\endt^v]) \in \{0,1,?\}^{H},
\] 
where $H=\endt^v-\init^v+1$ is the length of its evaluation horizon and 
$\imath[t],t=\init^v,\dots,\endt^v$ denotes its $t$-th element. 
\end{definition}
In the above definition, we count the first element of $\imath$ starting from $\imath[\init^v]$ rather than $\imath[1]$ to align with the absolute time instant from the perspective of the root node. This ensures consistency between the indexing of the satisfaction vector and the actual time instants being evaluated.  



For each node, its satisfaction vector is essentially determined by the   vectors of its children nodes. Through an inductive argument based on the tree structure, once the satisfaction vector of each $\mathcal{H}$-node is known, we can compute the satisfaction vectors of all non-leaf nodes. Therefore, we refer to the satisfaction vector of an $\mathcal{H}$-node as a \emph{basic vector}, as it serves as the foundation for \emph{inducing vectors} for other nodes. This leads to the following definitions.

\begin{definition}[\bf Basic  Vectors]\upshape
A \emph{basic set} of satisfaction vectors is a tuple of satisfaction vectors of the form
\[
I = (\imath^{v_1}, \dots, \imath^{v_h}) 
\in 
\{0, 1, ?\}^{H_1} \times \cdots \times \{0, 1, ?\}^{H_h},
\] 
where 
$V_\mathcal{H} = \{v_1, \dots, v_h\}$ is the set of all $\mathcal{H}$-nodes
and  for each $i = 1, \dots, h$, $H_i = \endt^{v_i} - \init^{v_i} + 1$ represents the length of the evaluation horizon of  $v_i$. 
\end{definition}

Hereafter, we will maintain basic vectors as time-updated information. For each time instant \( k = 0,1,\dots,T \), a basic set \( I_k = (\imath^{v_1}_k, \dots, \imath^{v_h}_k) \) must satisfy the following condition for each \( v \in V_\mathcal{H} \) and \( t = \init^v,\dots,\endt^v \):
\begin{itemize}
    \item 
    If \( t \!\geq\! k \), then \( \imath^v_k[t] \!=\; ? \) (unknown status for futures); and 
    \item 
    If \( t \!<\! k \), then \( \imath^v_k[t] \!\neq\; ? \) (determined status for pasts).
\end{itemize}
This reflects the  constraint that satisfaction status cannot be evaluated for future time instants, while for past instants the status must be resolved to either 0 (violated) or 1 (satisfied). We denote by  $\mathcal{I}_k$ the set of all potential basic vectors at instant $k$ satisfying the above two constraints.

\begin{definition}[\bf Induced  Vectors]\upshape \label{def-5}
Given a basic set of satisfaction vectors $I = (\imath^{v_1}, \dots, \imath^{v_h}) $,  
the \emph{induced set of satisfaction vectors} for the remaining nodes, denoted by  
$I^+=(\imath^{v_{h+1}},\dots, \imath^{v_m})$,  
is defined recursively by: 
for each $v\in \{v_{h+1},\dots,v_m\}$,   we have
\begin{itemize}
    \item 
    If $v\in V_{\wedge}$, then for each $t=\init^v,\dots,\endt^v$, we have
    \[
     \imath^v[t] = 
     \left\{
    \begin{array}{cl} 
     0 &\quad  \text{if } \exists v'\in \textsf{child}(v):   \imath^{v'}[t]=0 \\
     1  &\quad \text{if } \forall v'\in \textsf{child}(v):  \imath^{v'}[t]=1\\     
     ?  &\quad \text{otherwise}  \\
    \end{array}
\right..
    \]   
    \item 
    If $v\in V_{\mathbf{G}}$,  then for each $t=\init^v,\dots,\endt^v$, we have
    \[
     \imath^v[t] = 
     \left\{
    \begin{array}{cl} 
     0  &\quad  \text{if } \exists t'\in [a^v,b^v]:  \imath^{v'}[t+t']=0 \\
     1  &\quad \text{if }  \forall t' \in [a^v,b^v]:  \imath^{v'}[t+t']=1\\     
     ?  &\quad \text{otherwise}  \\
    \end{array}
\right..
    \]   
    where $v'$ is the unique child node of $v$.
    \item 
    If $v\in V_{\mathbf{U}'}$,  then for each $t=\init^v,\dots,\endt^v$, we have
    \[
     \imath^v[t] = 
     \left\{
    \begin{array}{cl} 
     0  &\quad  \text{if } \!\!\!\!\!\!\!\! 
     \begin{array}{cc}
          &   (\forall t'\in [a^v,b^v])(\imath^{v_r}[t+t']\!=\!0) \\
          &      \textbf{ or}\\
          & \left[\!\!\!\!\!\!\!\!\! 
     \begin{array}{ll}
      & (\forall t'\!\in\! [a^v,b^v]: \imath^{v_r}[t+t']\!=\!1)\\
      & (\exists t''\!\in\! [a^v,t'])(\imath^{v_l}[t+t'']\!=\!0)\\  
     \end{array}\!\!\!
     \right] 
     \end{array}\vspace{5pt}\\
     1  &\quad \text{if }  
      \left[\!\!\!\!\!\!\!\!\! 
     \begin{array}{ll}
      & (\exists t'\!\in\! [a^v,b^v]: \imath^{v_r}[t+t']\!=\!1)\\
      & (\forall t''\!\in\! [a^v,t'])(\imath^{v_l}[t+t'']\!=\!1)\\  
     \end{array}\!\!\!
     \right]\vspace{5pt}\\     
     ?  &\quad \text{otherwise}  \\
    \end{array}
\right.,
    \]   
    where $v_l$ and $v_r$ denotes the left and right child nodes of $v$, respectively.
\end{itemize}
\end{definition}

The intuition behind the above definition is essentially a reinterpretation of the semantics of STL formulae. Note that the induced vectors are well-defined because all nodes are ordered in the tree structure, where the basic vectors representing $\mathcal{H}$-nodes are leaves in the tree with no children. Therefore, given a basic set of satisfaction vectors, one can compute its induced vectors in a bottom-up manner, starting from the leaf nodes and progressing to the root node.

\section{Online Monitoring Algorithms}\label{sec:5}
\subsection{Evolution of Basic Vectors}
In order to track the progress of the STL formulae without storing the entire state trajectory, our approach is to maintain only the basic satisfaction vectors and update them recursively upon observing the new system state at each time instant. This recursive computation proceeds as follows:
\begin{itemize}
    \item 
    \textbf{Initialization: } 
    At time instant \(k = 0\), prior to observing any system state, the initial basic set is defined as:
    \[
    I_0 = (\imath^{v_1}_0, \imath^{v_2}_0, \dots, \imath^{v_h}_0), 
    \]
    where each \(\imath^{v_i}_0=(?,?,\dots,?)\) is a vector with all entries initialized to the uncertain status ``?". This reflects complete uncertainty about future satisfaction before any state observations are made.
    \item 
    \textbf{Online Update: } 
    At time instant \( k \), given the current basic satisfaction vector set 
    \( I_k = (\imath^{v_1}_k, \dots, \imath^{v_h}_k) \), 
    the update upon observing a new state \( x_k \) is defined as:
    \[
    I_{k+1} = (\imath^{v_1}_{k+1}, \dots, \imath^{v_h}_{k+1}) = \texttt{update}(I_k, x_k),
    \]
    where for each \( v_i \in \{v_1, \dots, v_h\} \) and \( t \in [\init^{v_i}, \endt^{v_i}] \), the updated satisfaction vector entry is:
\begin{equation}\label{eq:online_update} 
\imath^{v_i}_{k+1}[t] = 
\begin{cases}
\imath^{v_i}_k[t] & \text{if } t \neq k, \\
0 & \text{if } t = k \land x_k \notin \mathcal{H}^{v_i}, \\
1 & \text{if } t = k \land x_k \in \mathcal{H}^{v_i}.
\end{cases}
\end{equation}
Therefore, the uncertain status "?" in the basic vectors is resolved at time \(k\) based on the observed state \(x_k\). 
\end{itemize}

Let \(\mathbf{x}_{0:k-1} = x_0 x_1 \dots x_{k-1}\) be a state sequence of length \(k\). We denote by \(I(\mathbf{x}_{0:k-1})\) the basic set of satisfaction vectors reached recursively by the state sequence \(\mathbf{x}_{0:k-1}\) from \(I_0\).  
On the other hand, for any basic set \(I\), we define the set of state sequences of length \(k\) consistent with \(I\) as  
\[
\mathbf{x}_{0:k-1}^I = \{\mathbf{x}_{0:k-1} \in \mathcal{X}^k \mid I(\mathbf{x}_{0:k-1}) = I\}.
\]  
The following result demonstrates that the recursive computation of the basic set above indeed captures all task-relevant information from the complete state trajectory.

\begin{proposition}\label{prop-I}\upshape
Let $I$ be a basic set and $k$ be a time instant. 
For any two sequences 
$\mathbf{x}_{0:k-1}',\mathbf{x}_{0:k-1}''\in \mathbf{x}_{0:k-1}^I$
consistent with $I$, 
and any future sequence $\mathbf{x}_{k:T}=x_{k}x_{k+1}\cdots x_T$, we have 
\begin{equation}
\mathbf{x}_{0:k-1}'\mathbf{x}_{k:T}\models \Phi 
\quad \Leftrightarrow \quad 
\mathbf{x}_{0:k-1}''\mathbf{x}_{k:T}\models \Phi. 
\end{equation}
\end{proposition}
\begin{proof}
    Due to space constraint, all proofs in the paper are omitted
    and they are available in the  supplementary materials {\footnotesize\url{https://xiangyin.sjtu.edu.cn/25CDC-STL.pdf}}
\end{proof}\vspace{3pt}

Based on the above result, it is meaningful to write\vspace{-3pt}  
\begin{equation}  
    \mathbf{x}_{0:k-1}^I \mathbf{x}_{k:T} \models \Phi  
\end{equation}  
whenever $\mathbf{x}_{0:k-1}' \mathbf{x}_{k:T} \models \Phi$ holds for some (or equivalently, for all, due to Proposition~\ref{prop-I}) $\mathbf{x}_{0:k-1}' \in \mathbf{x}_{0:k-1}^I$.

Given a basic set of satisfaction vectors \( I \), its induced set of satisfaction vectors is defined as \( I^+ = (\imath^{v_{h+1}}, \dots, \imath^{v_m}) \).  
We call the satisfaction vector corresponding to the root node \( v_{\text{root}} \) in \( I^+ \) the \emph{root satisfaction vector}, denoted by \( \imath^{\text{root}}_I \in I^+ \).   
With lose of generality, we assume that the root node is always a $\wedge$-node, whose evaluation horizon is $[0,0]$. 
The following result demonstrates that this root vector fully captures the satisfaction status of the entire STL formula.  
\begin{proposition}\label{prop-2}\upshape
For any state sequence $\mathbf{x}_{0:k}$, we have 
\begin{equation}
  \mathbf{x}_{0:k}\models \Phi  \quad\Leftrightarrow \quad   \imath^{\text{root}}_{I(\mathbf{x}_{0:k})}[0]=1.
\end{equation} 
\end{proposition} 
For simplicity, we will write $\imath^{\text{root}}_{I(\mathbf{x}_{0:k})}= 0,1,?$ whenever $\imath^{\text{root}}_{I(\mathbf{x}_{0:k})}[0]=0,1,?$ as the root vector only has one element. 

\subsection{Model-Predictive Monitor}
 By recursively maintaining the basic set \( I_k = I(\mathbf{x}_{0:k-1}) \), we can draw the following conclusions about the current status of the entire specification formula:

\begin{itemize}
    \item 
    If \( \imath^{\text{root}}_{I_k} = 0 \), the specification has been violated, and the monitor should output \( \mathcal{M}(\mathbf{x}_{0:k-1}) = \texttt{vio} \);
    \item If \( \imath^{\text{root}}_{I_k} = 1 \), the specification has been satisfied, and no further monitoring is required.
\end{itemize}
These conclusions are independent of the system dynamics, as the satisfaction status is fully determined in these cases. However, when \( \imath^{\text{root}}_{I_k} = ? \), both \(\texttt{vio}\) and \(\texttt{feas}\) remain possible outcomes, requiring analysis of the system dynamics. This leads to the following key definition.
  
\begin{definition}[\bf $I$-Determined Feasible Sets] \upshape
Let   $\Phi$ be an STL formula,  $k \in [0, T]$ be a time instant and 
$I$  be a basic set. 
Then the \emph{$I$-determined feasible set} at instant $k$, 
denoted by $X_k^I \subseteq \mathcal{X}$, is the set of states from which there exists a solution $\mathbf{u}_{k:T-1}$ that satisfies $\Phi$ 
given the current basic set $I$, i.e.,
\begin{equation}
    X_k^I=\left\{x_k \hspace{-0.05cm} \in \hspace{-0.05cm} \mathcal{X}\left|\begin{matrix}\exists \mathbf{u}_{k:T-1} \hspace{-0.1cm} \in\mathcal{U}^{T-k} \\\mathrm{s.t.~}\mathbf{x}_{0:k-1}^I x_k\xi_f(x_k,\mathbf{u}_{k:T-1})\models{\Phi}\end{matrix}\right.\right\}
\end{equation} 
\end{definition}
 
Based on the above notion, 
now we present our main online  monitoring algorithm as shown in Algorithm 1. 
The algorithm initializes with the time instant $k=0$  and the initial basic set $I=I_0$ (line 1). 
The monitoring process iterates as long as the STL formula remains unresolved  (line 2). At each iteration, the current state $x_k$ is read (line 3) and checked against the feasible set $X_k^I$. 
If $x_k \not\in X_k^I$, the algorithm immediately terminates with a violation decision (lines 4-6). Otherwise, the monitoring decision is set to $\mathcal{M}=\texttt{feas}$, and the basic set $I$ is recursively updated by $\texttt{update}(I, x_k)$ (lines 7-9). 
The time instant is then increased (line 10), and the loop continues until either a violation is detected or the STL formula is satisfied.

The correctness of Algorithm is established as follows. 
\begin{theorem}\upshape
The online monitor $\mathcal{M}$ defined by Algorithm~1 indeed solves Problem~1. 
\end{theorem}



\begin{algorithm2e}[t]
    \caption{Online Monitoring algorithm}
    \renewcommand{\algorithmcfname}{Upward update111}
    \label{alg:prune_composed_DFA}
    \KwIn{Feasible sets $X_k^I$ computed offline}
    \KwOut{Online monitoring decision $\mathcal{M}(\mathbf{x}_{0:k})$}
    $k \gets 0$, $I \gets I_0$\\
    \While{$\imath^{\text{\upshape root}}_I$ $\neq 1$}
    {
        observe a new current state $x_k$\\
        \If{$x_k \notin X_k^I$}
            {$\mathcal{M}$ issues ``\texttt{vio}"\\
            \Return {``$\Phi$ is violated"}}
        \Else
            {$\mathcal{M}$ issues ``\texttt{feas}"\\
            $I\gets \texttt{update}(I, x_k)$}
        $k \gets k + 1$\\
    }
\end{algorithm2e}

\subsection{Offline Computation of Feasible Sets}
The proposed online monitoring algorithm utilizes pre-computed all feasible sets $X_k^I$ at each time instant. 
This subsection details the offline computation of these sets.

Note that, at each time instant $k$, we only need to consider those basic vectors for which the entire STL formula $\Phi$ has not been violated. 
Therefore, we define
\[
\mathbb{I}_k = \{I\in \mathcal{I}_k \mid \imath^{\text{root}}_I \not= 0\}
\]
the set of \textbf{feasible basic vectors} at time instant $k$.  

Furthermore,  the basic vector set cannot be updated arbitrarily in the next time instant as it should be consistent with the existing history. 
Therefore, 
let 
$I = (\imath^{v_1}, \dots, \imath^{v_h})\in\mathbb{I}_k$ 
be a basic vector set for time instant $k$
and 
$I_{s} = (\imath_s^{v_1}, \dots, \imath_s^{v_h})\in\mathbb{I}_{k+1}$
be a basic vector set for time instant $k+1$. 
We say $I_s\in\mathbb{I}_{k+1}$ is a \textbf{successor basic set} of $I\in\mathbb{I}_k$ at instant $k$ 
if
\[
\forall v\in V_{\mathcal{H}}, \forall t\in [\init^v,\min(k-1,\endt^v)]:  \imath^{v}[t]=\imath_s^{v}[t].
\]
We denote by $\operatorname{succ}(I, k)\subseteq \mathbb{I}_{k+1}$ the set of all successor basic sets of $I$ at instant $k$.
Moreover, when the basic set evolves from $I$ to $I_s$ at instant $k$, 
it must reach a state $x_k$ consistent to the update rule. 
We define the region that $x_k$ needs to be in as \textit{consistent region}.

\begin{definition}[\bf Consistent Regions]\upshape 
Let $I = (\imath^{v_1}, \dots, \imath^{v_h})\in \mathbb{I}_k$ be a basic vector at instant $k$, and $I_{s} = (\imath_s^{v_1}, \dots, \imath_s^{v_h})\in \operatorname{succ}(I, k)$ be a successor basic set of $I$. In order to trigger the evolution of the basic set from $I$ to $I_s$ at instant $k$, the system should be in the consistent region $H_k(I, I_s)$, which is defined by
\begin{equation}
H_k(I, I_s) = \bigcap_{i=1}^h {H_i}_k,
\end{equation}
where 
\begin{equation*}
{H_i}_k = 
\begin{cases} 
\mathcal{X} \setminus \mathcal{H}^{v_i} & \text{if } \imath^{v_i}_s[k] = 0 \\
\mathcal{H}^{v_i} & \text{if } \imath^{v_i}_s[k] = 1 \\
\mathcal{X} & \text{if } k > \textup{$\endt^{v_i}$} \;\text{or}\; k < \textup{$\init^{v_i}$}
\end{cases}
\end{equation*}
\end{definition}

Finally, in order to ensure that the prefix signal is always feasible, 
the feasible set $X^I_k$ needs to be able to reach region $X^{I_s}_{k+1}$ in one step. 
This is formalized by the \textit{one-step feasible set} defined as follows.

\begin{definition}[\bf One-Step Feasible Set] \upshape
Let $\mathcal{S} \subseteq \mathcal{X}$ be a
set of states representing the ``target region”. Then the
one-step feasible set of $\mathcal{S}$ is defined by
\begin{equation}
    \Upsilon(\mathcal{S})=\{x\in\mathcal{X}\mid\exists u\in\mathcal{U}\;\text{s.t.}\;f(x,u)\in\mathcal{S}\}.
\end{equation}
\end{definition}

We define $\mathbb{X}_k = \{X^I_k \mid I \in \mathbb{I}_k\}$ as the set of all feasible sets for instant $k$. 
Then our offline objective is to compute all possible  $\mathbb{X}_0,\mathbb{X}_1,\dots, \mathbb{X}_T$, which are used as a look-up table during the online monitoring process.
In terms of our computation of feasible regions, if the system is evolving from $I$ to $I_s$ and maintains the satisfiability of $I_s$ from instant $k+1$, then we know that the system should be in region $H_k(I, I_s) \cap \Upsilon(X_{k+1}^{I_s})$ at instant $k$. However, the basic set $I_s$ for the next instant depends on the current state of the system. Therefore, to compute $X_k^I$, we need to consider all possible successor sets $I_s \in \operatorname{succ}(I, k)$, and take the union of these regions. This is formalized by the following equation
\begin{equation}
X_k^I = \bigcup_{I_s \in \operatorname{succ}(I, k)} \left( H_k(I, I_s) \cap \Upsilon(X_{k+1}^{I_s}) \right).
\end{equation}

Based on the above equation, we compute all feasible sets across the entire horizon using a backward recursion procedure, as formalized in Algorithm 2. 
The algorithm initializes by setting the terminal feasible set \(X^{I}_{T+1} \gets \mathbb{R}^n\) for all \(I \in \mathbb{I}_{T+1}\) (lines 3–4). 
The recursion proceeds backward in time from \(k = T\) to \(k = 0\). If the STL formula is already satisfied (i.e., $\imath^{\text{root}}_I = 1$), then set $X_k^I \leftarrow \mathbb{R}^n$.
At each time step \(k\), the feasible set \(X^I_k\) is computed for every basic set \(I \in \mathbb{I}_k\) through two  operations: 
\begin{itemize}
    \item 
    Intersection with \(H_k(I, I_s)\) and \(\Upsilon(X^{I_s}_{k+1})\), and  
    \item 
     Union over all successor sets \(I_s \in \text{succ}(I, k)\) to ensure completeness. 
\end{itemize}
The resulting sets are aggregated into \(\mathbb{X}_k\), which   captures all feasible states at time \(k\) (lines 5–13).

\begin{algorithm2e}[t]
    \caption{Offline Computations of Feasible Sets}
    \renewcommand{\algorithmcfname}{Upward update111}
    \label{alg:prune_composed_DFA}
    \KwIn{STL formula $\Phi$}
    \KwOut{All potential sets $\{\mathbb{X}_k : k \in [0, T]\}$}
    \For{$k \in [0, T]$} 
        {$\mathbb{X}_k \gets \emptyset$}
    \For{$I\in \mathbb{I}_{T+1}$}
        {$X_{T+1}^I \gets \mathbb{R}^n$}
    $k \gets T$\\
    \While{$k \geq 0$}
    {
        \For{$I \in \mathbb{I}_k$}
        {
            \If{$\imath^{\text{\upshape root}}_I = 1$}
                {$X_k^I \gets \mathbb{R}^n$}
            \Else
                {$X_k^I\hspace{-0.08cm} \gets\hspace{-0.08cm} \bigcup_{I_s \in \operatorname{succ}(I, k)}\hspace{-0.08cm} \left( H_k(I,\hspace{-0.05cm} I_s) \hspace{-0.07cm}\cap\hspace{-0.07cm} \Upsilon(X_{k+1}^{I_s}) \right)$}
        }
        $\mathbb{X}_k \gets \mathbb{X}_k \cup \{X_k^I\}$\\
        $k \gets k - 1$
    }
\end{algorithm2e}



\section{Case Studies for Online Monitoring}\label{sec:6}
In this section, we illustrate our online monitoring algorithm with two cases. We implemented the above methods in Python language.\footnote{Our codes are available at {\url{https://github.com/sjtu-hantao/MPM4STLnested}}, where more details on the computations can be found.  }





\subsection{Building Temperatures}
We consider the problem of monitoring the temperature of a single zone building whose dynamic is \cite{jagtap2020formal}  
\[
x_{k+1}=x_k+\tau_s(\alpha_e(T_e-x_k)+\alpha_H(T_h-x_k)u_k),
\]
where \( x_k \in \mathcal{X} = [0, 45] \) is the zone temperature (°C) at step \( k \), 
\( u_k \in \mathcal{U} = [0, 1] \) is the heater valve’s normalized position, 
and \( \tau_s = 1 \) min is the sampling interval. 
The model parameters include the heater temperature \( T_h = 55^\circ \mathrm{C} \), ambient temperature \( T_e = 0^\circ \mathrm{C} \), and heat transfer coefficients \( \alpha_e = 0.06 \) (environmental) and \( \alpha_H = 0.08 \) (heater).

The temperature control system aims to maintain the zone temperature within the comfortable range of 20°C–25°C at all times, ensuring that any temperature deviation is corrected within 5 minutes during any 10-minute window. This requirement can be specified using the STL formula:
\[
\Phi=\mathbf{G}_{[0,10]}(\mathbf{F}_{[0,5]}x_k\in[20,25]).
\]
Fig.~\ref{fig:case-1} presents two temperature trajectories that are identical until time step $k=13$. 
For the black signal, the condition $x_{13}\in X^I_{13}$ holds, indicating that the control task remains feasible. 
However, for the red signal, $x_{13}\not\in X^I_{13}$, which means the specification will inevitably be violated. 
That is, no control input exists that could bring the temperature within the required range within the specified time window.

\begin{figure}[t]
    \centering
    \includegraphics[scale=0.28]{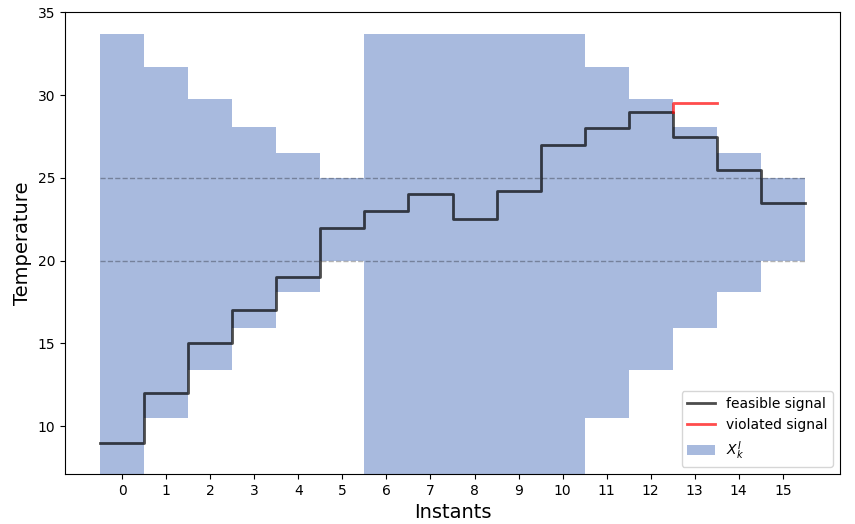}
    \caption{Two trajectories temperature control system.}
    \label{fig:case-1} 
\end{figure}
 
\subsection{Autonomous Robots}
We consider a simple autonomous robot whose dynamic model is given as follows
$$x_{k+1}=
\begin{bmatrix}
1 & 0 \\
0 & 1
\end{bmatrix}x_k+
\begin{bmatrix}
\tau_s & 0\\
0 & \tau_s
\end{bmatrix}u_k,$$
where state 
$ x_k \in \mathcal{X} = [0, 12]\times[0, 12] $ denotes the position of the robot at instant $ k $,
control input $ u_k \in \mathcal{U} = [-1, 1]\times[-1, 1] $ is the speed of the robot,
and $ \tau_s = 1s$  is the sampling time.
  
The objective of the robot is to  patrol both region $A_1$ and region $A_2$  in 6 seconds and 
stay in $A_2$ for at least 2 seconds. 
This task can be described by the following STL formula
$$\Phi=\mathbf{F}_{[0,6]}A_1\wedge\mathbf{F}_{[0,6]}(\mathbf{G}_{[0,2]}A_2),$$
where $A_{1}=(x{\in}[3,5])\wedge(y{\in}[3,5])$ and $A_{2}=(x{\in}[6,8])\wedge(y{\in}[6,8])$.

Let us consider two trajectories of the robot shown in Fig.~\ref{fig:case-2}. 
For the sake of clarity, in each figure, we only draw one feasible set at a certain instant. 
In the left figure, 
since $x_{0}\not\in X^I_{0}$, 
the monitor can claim initially that the task cannot be satisfied as it is too far away from the target regions. 
In the right figure, 
for $k=0,1,2,3$, the trajectory stays within the feasible region
and the monitor issues ``\texttt{feas}". 
Yet, we have $x_{4}\not\in X^I_{4}$. 
Therefore, the monitor issues ``\texttt{vio}" at $k=4$.

\begin{figure}[t]
    \centering
    \includegraphics[scale=0.39]{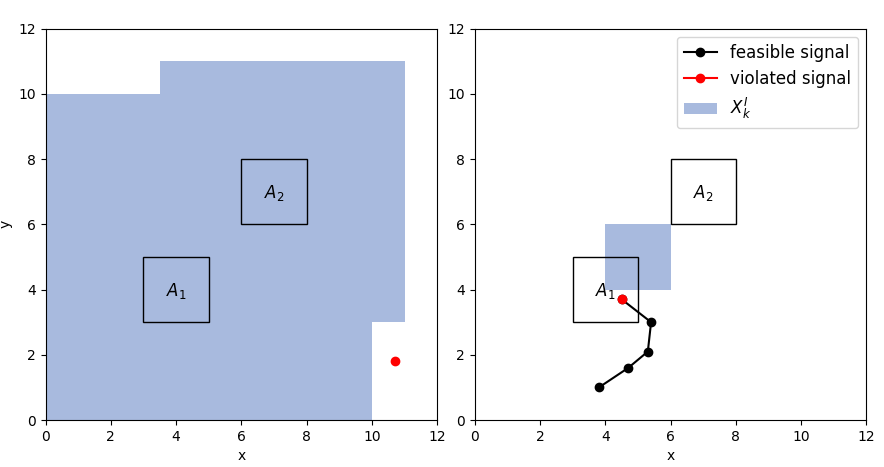}
    \caption{Two  trajectories that violate at different instants.}
    \label{fig:case-2}
\end{figure}



\section{Conclusions}\label{sec:7}
In this paper, we present a new model-based online monitoring algorithm for signal temporal logic specifications with nested temporal operators. 
Our approach utilizes the syntax tree as the fundamental information structure to characterize and dynamically update all relevant information for STL task progression.
The online monitoring process is highly efficient as it requires only basic set updates and membership checks. 
While  implemented in an offline manner, the feasible set computation remains challenging for highly nonlinear systems with long time horizons.
In the future,  we plan to enhance the framework by developing more efficient set representation techniques to optimize the offline computation of feasible sets.

\addtolength{\textheight}{-0cm}   



\bibliographystyle{plain} 
\bibliography{mpm}

\newpage
Let $v\in V$ be a node in  tree $\mathcal{T}$. 
The \emph{sub-tree} induced from node $v$ is a new rooted tree
$\mathcal{T}^v = (V', L, E', v_{\text{root}}') $, 
where $V'=\{v'\in V\mid v'\preceq v\}$ is the set of descendants of $v$ including itself, 
$E'\subseteq V' \times L \times V'$ is the restriction of $E$ onto $V'$, 
and $v_{\text{root}}'=v$  is the root of the sub-tree.

Let $\Phi$ be an STL formula and $\mathcal{T}_\Phi$ be its syntax tree. 
The induced sub-tree from  any node $v\in V_\Phi$ is also a syntax tree 
We use $\mathcal{T}^{v}$ to denote the subtree with $v$ as the root node. The subtree $\mathcal{T}^{v}$ represents a subformula $\Phi^v$ of the original STL formula $\Phi$.

To prove Proposition 1 and Proposition 2, we prove Proposition 3 at first. Then we use Proposition 3 to prove Proposition 2 and prove Proposition 1 at last.
\begin{proposition}\label{prop-3}\upshape
For any state sequence $\mathbf{x}_{0:k}$ and subformula $\Phi^{v}$ we have 
$$
  (\mathbf{x}_{0:k},t)\models \Phi^{v}  \quad\Leftrightarrow \quad   \imath^{v}_{I(\mathbf{x}_{0:k})}[t]=1.
$$
\end{proposition}
\begin{proof}
    Prove this proposition recursively by mathematical induction. For the case of leaf node $\Phi^v = x\in\mathcal{H}$:\\
    $\Rightarrow$: When $(\mathbf{x}_{0:k},t)\models \Phi^{v}$, there must be $x_t\in\mathcal{H}$. According to ~\eqref{eq:online_update}, $\imath^{v}_{I(\mathbf{x}_{0:k})}[t]=1$.\\
    $\Leftarrow$: When $\imath^{v}_{I(\mathbf{x}_{0:k})}[t]=1$, which also means $x_t\in\mathcal{H}$. So $(\mathbf{x}_{0:k},t)\models \Phi^{v}$.
    \\The proposition holds for leaf nodes. Assuming it always holds for child nodes, now prove the case of other nodes.
    \begin{enumerate}
        \item $\Phi^v = \Phi^{v^c_1}\wedge\Phi^{v^c_2}\wedge\dots\wedge\Phi^{v^c_n}$:\\
        $\Rightarrow$: When $(\mathbf{x}_{0:k},t)\models \Phi^{v}$, there must be $(\mathbf{x}_{0:k},t)\models \Phi^{v^c_1}, (\mathbf{x}_{0:k},t)\models \Phi^{v^c_2},\dots, (\mathbf{x}_{0:k},t)\models \Phi^{v^c_n}$. So $\forall v^c_i\in \textsf{child}(v)$, $\imath^{v^c_i}_{I(\mathbf{x}_{0:k})}[t]=1$. Then $\imath^{v}_{I(\mathbf{x}_{0:k})}[t]=1$ according to Definiton ~\ref{def-5}.
        \\$\Leftarrow$: Reverse the proof of $\Rightarrow$.
        
        \item $\Phi^v = \mathbf{G}_{[a,b]}\Phi^{v^c}$:\\
        $\Rightarrow$: When $(\mathbf{x}_{0:k},t)\models \Phi^{v}$, there must be $\forall t'\in [t+a,t+b], (\mathbf{x}_{0:k},t')\models \Phi^{v^c}$. So $\forall t'\in [t+a,t+b], \imath^{v^c}_{I(\mathbf{x}_{0:k})}[t']=1$. So $\forall v^c_i\in \textsf{child}(v)$, $\imath^{v^c_i}_{I(\mathbf{x}_{0:k})}[t]=1$. Then $\imath^{v}_{I(\mathbf{x}_{0:k})}[t]=1$ according to Definiton ~\ref{def-5}.
        \\$\Leftarrow$: Reverse the proof of $\Rightarrow$.

        \item $\Phi^v = \Phi^{v^{lc}}\mathbf{U'}_{[a,b]}\Phi^{v^{rc}}$:\\
        $\Rightarrow$: When $(\mathbf{x}_{0:k},t)\models \Phi^{v}$, there must be $\exists t'\in [t+a,t+b], (\mathbf{x}_{0:k},t')\models \Phi^{v^{rc}}$ i.e. $\imath^{v^{rc}}_{I(\mathbf{x}_{0:k})}[t']=1$ \:and\: $\forall t''\in[t+a, t']$ such that $(\mathbf{x}_{0:k},t'')\models \Phi^{v^{lc}}$ i.e. $\imath^{v^{lc}}_{I(\mathbf{x}_{0:k})}[t'']=1$. So $\imath^{v}_{I(\mathbf{x}_{0:k})}[t]=1$ according to Definiton~\ref{def-5}.
        \\$\Leftarrow$: Reverse the proof of $\Rightarrow$.
    \end{enumerate}
    According to the above, we prove $(\mathbf{x}_{0:k},t)\models \Phi^{v}  \Leftrightarrow  \imath^{v}_{I(\mathbf{x}_{0:k})}[t]=1$. 
\end{proof}

\vspace{10pt}
\noindent\textbf{Proposition 2.} For any state sequence $\mathbf{x}_{0:k}$, we have 
$$
  \mathbf{x}_{0:k}\models \Phi  \quad\Leftrightarrow \quad   \imath^{\text{root}}_{I(\mathbf{x}_{0:k})}[0]=1.
$$
\begin{proof}
    As a special case of  Proposition~\ref{prop-3}, when $v=v_\text{root}$ and $t=0$, $\mathbf{x}_{0:k}\models \Phi \Leftrightarrow   \imath^{\text{root}}_{I(\mathbf{x}_{0:k})}[0]=1$, which means the root vector fully captures the satisfaction status of the entire STL formula.
\end{proof}

\vspace{10pt}
\noindent\textbf{Proposition 1.} Let $I$ be a basic set and $k$ be a time instant. 
For any two sequences 
$\mathbf{x}_{0:k-1}',\mathbf{x}_{0:k-1}''\in \mathbf{x}_{0:k-1}^I$
consistent with $I$, 
and any future sequence $\mathbf{x}_{k:T}=x_{k}x_{k+1}\cdots x_T$, we have 
$$
\mathbf{x}_{0:k-1}'\mathbf{x}_{k:T}\models \Phi 
\quad \Leftrightarrow \quad 
\mathbf{x}_{0:k-1}''\mathbf{x}_{k:T}\models \Phi. 
$$

\begin{proof}
         According to Proposition~\ref{prop-2}, we know that $\mathbf{x}_{0:k-1}'\mathbf{x}_{k:T}\models \Phi \Leftrightarrow \imath^{v_{\text{root}}}_{I(\mathbf{x}_{0:k-1}'\mathbf{x}_{k:T})}[0]=1$. For $\mathbf{x}_{0:k-1}',\mathbf{x}_{0:k-1}''\in \mathbf{x}_{0:k-1}^I$, we have $I(\mathbf{x}_{0:k-1}'\mathbf{x}_{k:T}) = I_(\mathbf{x}_{0:k-1}''\mathbf{x}_{k:T})$. So there must be $\imath^{v_{\text{root}}}_{I(\mathbf{x}_{0:k-1}'\mathbf{x}_{k:T})}[0]=\imath^{v_{\text{root}}}_{I(\mathbf{x}_{0:k-1}''\mathbf{x}_{k:T})}[0]=1$. Then due to Proposition~\ref{prop-2}, we have $\mathbf{x}_{0:k-1}''\mathbf{x}_{k:T}\models \Phi$. The proof above also holds in reverse. So there is $\mathbf{x}_{0:k-1}'\mathbf{x}_{k:T}\models \Phi \Leftrightarrow \mathbf{x}_{0:k-1}''\mathbf{x}_{k:T}\models \Phi. $
\end{proof}
 
\end{document}